\documentclass[10pt, a4paper]{article}
\usepackage{amsmath, amsfonts, amssymb, amscd, amsthm, mathdots, mathrsfs, dsfont}
\usepackage{titlesec}
\usepackage{fancyhdr}
\usepackage{fullpage}
\usepackage[colorlinks, urlcolor=black]{hyperref}
\hypersetup{citecolor=blue, linkcolor=blue}
\usepackage{mathabx}

\newcommand{\A}{\mathbb{A}}

\newcommand{\CC}{\mathbb{C}}
\newcommand{\F}{\mathbb{F}}
\newcommand{\Frob}{\mathrm{Frob}}
\newcommand{\GL}{\mathrm{GL}}
\newcommand{\GO}{\mathrm{GO}}
\newcommand{\GSp}{\mathrm{GSp}}

\newcommand{\PSL}{\mathrm{PSL}}
\newcommand{\PGL}{\mathrm{PGL}}

\newcommand{\POM}{\mathrm{P}\Omega}

\newcommand{\QQ}{\mathbb{Q}}

\newcommand{\RR}{\mathbb{R}}
\newcommand{\SL}{\mathrm{SL}}
\newcommand{\Sp}{\mathrm{Sp}}
\newcommand{\Spin}{\mathrm{Spin}}

\newcommand{\ZZ}{\mathbb{Z}}

\DeclareMathOperator{\Aut}{Aut}

\DeclareMathOperator{\Gal}{Gal}
\DeclareMathOperator{\Hom}{Hom}
\DeclareMathOperator{\HT}{HT}
\DeclareMathOperator{\Inert}{Inert}
\DeclareMathOperator{\im}{Im}

\DeclareMathOperator{\Sym}{Sym}
\DeclareMathOperator{\SO}{SO}

\DeclareMathOperator{\Tr}{Tr}

\titleformat{\section}[hang]
{\normalfont\filright\large}{\thesection . }{0pt}
{\upshape\bfseries}

\titleformat{\subsection}[hang]
{\itshape}{\thesubsection \ - }{0pt}
{}

\newtheorem{theorem}{Theorem}[section]
\newtheorem{proposition}[theorem]{Proposition}

\newtheorem{corollary}[theorem]{Corollary}

\theoremstyle{definition}

\theoremstyle{remark}
\newtheorem{remark}{Remark}

\title{Galois representations with large image in the global Langlands correspondence}
\author{\small ADRI$\acute{\mbox{A}}$N ZENTENO}
\date{\today}
\begin{document}

\maketitle

\begin{abstract}

The global Langlands conjecture for $\text{GL}_n$ over a number field $F$ predicts a correspondence between certain algebraic automorphic representations $\pi$ of $\text{GL}_n(\mathbb{A}_F)$ and certain families $\{ \rho_{\pi,\ell} \}_\ell$ of $n$-dimensional $\ell$-adic Galois representations of $\text{Gal}(\overline{F}/F)$. In general, it is expected that the image of the residual Galois representation $\overline{\rho}_{\pi,\ell}$ of $\rho_{\pi,\ell}$ should be as large as possible for almost all primes $\ell$, unless there is an automorphic reason for the image to be small.

In this paper, we study the images of certain compatible systems of Galois representations $\{\rho_{\pi,\ell} \}_\ell$ associated to regular algebraic, polarizable, cuspidal automorphic representations $\pi$ of $\text{GL}_n(\mathbb{A}_F)$ by using only standard techniques and currently available tools (e.g., Fontaine-Laffaille theory, Serre's modularity conjecture, classification of the maximal subgroups of Lie type groups, and known results about irreducibility of automorphic Galois representations and Langlands functoriality). In particular, when $F$ is a totally real field and $n$ is an odd prime number $\leq 293$, we prove that (under certain automorphic conditions) the images of the residual representations $\overline{\rho}_{\pi,\ell}$ are as large as possible for infinitely many primes $\ell$. In fact, we prove the large image conjecture (i.e., large image for almost all primes $\ell$) when $F=\mathbb{Q}$ and $n=5$.

\end{abstract}

\section{Introduction}\label{sec:intro}

Let $F$ be a number field, $\A_F$ be the ring of adeles of $F$ and $G_F:=\Gal(\overline{F}/F)$ be the absolute Galois group of $F$.
The global Langlands conjecture for $\GL_n$ (complemented by Clozel, Fontaine and Mazur) predicts a correspondence between certain algebraic automorphic representations of $\GL_n(\A_F)$ and certain $n$-dimensional $\ell$-adic Galois representations of $G_F$.  
In particular, in the automorphic to Galois direction, we can formulate a more concrete conjecture due to Clozel \cite{Clo90} as follows.
Let $\pi = \otimes '_v \pi_v$ be an algebraic cuspidal automorphic representation of $\GL_n(\A_F)$. According to Clozel,  for each prime $\ell$ (and a fixed field isomorphism $\iota_\ell: \overline{\QQ}_\ell \cong  \CC$), there should exist a Galois representation
\[
\rho_{\pi,\ell} : G_F \longrightarrow \GL_n(\overline{\QQ}_\ell)
\] 
such that for almost all finite places $v$ of $F$ (i.e., all but finitely many), $\rho_{\pi,\ell}$ is unramified and 
the characteristic polynomial of $\rho_{\pi,\ell}(\Frob_v)$ is described (up to normalization) by the Satake parameters of $\pi_v$.
There is also a version of this conjecture for more general reductive groups, but its formulation requires some care (see  \cite{BG14} for details). 

In general, Clozel's conjecture is still out of reach, but there are favorable automorphic circumstances where it can be proven. For example, when $\pi$ is regular algebraic and $F$ is a totally real or CM field, Harris-Lan-Taylor-Thorne \cite{HLTT16} and Scholze \cite{Sch15} have recently been able to attach a family of $\ell$-adic Galois representations $\{ \rho_{\pi,\ell} \}_{\ell}$ to $\pi$  
satisfying the desired ramification and compatibility conditions. 
Of particular interest to us is a less recent but more precise result due to Clozel, Kottwitz, Harris, Taylor and several others \cite{BLGGT14}, 
which allows us to attach a compatible system of $\ell$-adic Galois representations $\{ \rho_{\pi,\ell} \}_{\ell}$ to a regular algebraic, polarizable, cuspidal automorphic representation $\pi$ of $\GL_n(\A_F)$ when $F$ is a totally real or CM field. 

Assuming Clozel's conjecture, the \emph{irreducibility conjecture} asserts that the cuspidality of $\pi$ should imply the irreducibility of $\rho_{\pi,\ell}$ for all primes $\ell$ \cite{Ram08}. 
However, even under the polarizability hypothesis, the irreducibility of the representations $\rho_{\pi,\ell}$ associated to $\pi$ is known in very few cases.
More precisely, assuming $\pi$ polarizable and $F$ totally real or CM, the irreducibility conjecture is only known for $n = 2$ \cite{Mok14} \cite{Rib77} \cite{Tay94} \cite{Tay95}, $n=3$ \cite{BR92} \cite{CG13} \cite{Hui23}, and for arbitrary $n$, if $\pi$ satisfies certain local conditions \cite{HL25} \cite{TY07}. 
In addition, weak versions of this conjecture are known for more general polarizable automorphic representations over totally real or CM fields. For example: if $n \leq 6$, it is known that the Galois representations $\rho_{\pi,\ell}$ associated to $\pi$ are irreducible for almost all primes $\ell$ \cite{Dai25} \cite{DWW} \cite{Hui22} \cite{Hui23} \cite{Ram13} \cite{Xia19} and, if $n$ is arbitrary and $\pi$ is extremely regular, it can be proved that the Galois representations $\rho_{\pi,\ell}$ are irreducible for a set of primes $\ell$ of Dirichlet density one \cite{BLGGT14}. The best known result, without extra conditions on $\pi$, is that the Galois representations $\rho_{\pi,\ell}$ are irreducible for a positive Dirichlet density set of primes $\ell$ \cite{PT15}. Finally, we remark that the irreducibility conjecture has recently been proven, without the polarizability condition, when $n=3$ and $F$ is a totally real field \cite{BH24}.

The irreducibility of the $\ell$-adic Galois representation $\rho_{\pi,\ell}$ is closely related to the irreducibility of the residual representation 
\[
\overline{\rho}_{\pi,\ell}: G_F \longrightarrow \GL_n(\overline{\F}_\ell)
\] 
which is defined as the semi-simplification of the reduction of $\rho_{\pi,\ell}$. For example, in \cite{Rib75}, 
Ribet proves that the $\ell$-adic Galois representations $\rho_{\pi_f,\ell}: G_\QQ \rightarrow \GL_2(\overline{\QQ}_\ell)$, associated to an automorphic representation $\pi_f$ of $\GL_2(\A_\QQ)$ coming from a level one cusp form $f$, are irreducible for all $\ell$ as a consequence of the irreducibility of $\overline{\rho}_{\pi_f,\ell}$ for almost all  $\ell$. In fact, in his paper, Ribet proved a bit more. He proved that, if $f$ does not have complex multiplication then the image of $\overline{\rho}_{\pi_f,\ell}$ contains $\SL_2(\F_\ell)$ for almost all primes $\ell$. This result, is part of a folklore conjecture known as the \emph{large image conjecture}. 
Roughly speaking, this conjecture predicts that (the residual representations $\overline{\rho}_{\pi,\ell}$ are not only irreducible for almost all $\ell$ but) the images of the residual representations $\overline{\rho}_{\pi,\ell}$ associated to algebraic cuspidal automorphic representations $\pi$ of $\GL_n(\A_F)$ should be as large as possible for almost all primes $\ell$, unless there is an automorphic reason for it does not happen. In Ribet's result, such automorphic reason is complex multiplication. For the $\GL_2$ case, Ribet's result was extended to automorphic representations coming from classical cusp forms of arbitrary level by Momose \cite{Mo81} and Ribet \cite{Ri85}, to automorphic representations coming from Hilbert cusp forms by Dimitrov \cite{Dim05}, and to automorphic representations coming from Bianchi modular forms by Conti, Lang and Medvedovsky \cite{CLM23}.

On the other hand, by using the irreducibility of the $\ell$-adic Galois representations $\rho_{\pi,\ell}$ associated to regular algebraic, polarizable, cuspidal automorphic representations $\pi$ of $\GL_4(\A_F)$ (for a set of primes $\ell$ of Dirichlet density one \cite{CG13} \cite{CG16}), combined with certain results on residual irreducibility of compatible systems \cite{BLGGT14} \cite{PSW18} and Langlands functoriality from $\GSp_4$ to $\GL_4$ \cite{AS06}, Dieulefait-Zenteno \cite{DZ20} and Weiss \cite{Wei18} proved that the images of the residual Galois representations $\overline{\rho}_{\pi_f,\ell}: G_F \rightarrow \GSp_4(\overline{\F}_\ell)$ associated to certain genus two Hilbert-Siegel modular forms $f$ (which give rise to cuspidal automorphic representations $\pi_f$ of $\GSp_4(\A_F)$) contain $\Sp_4(\F_\ell)$ for a set of primes $\ell$ of Dirichlet density one. In this case, cuspidal automorphic representations $\pi_f$ of $\GSp_4(\A_F)$ being: CAP, endoscopic lifts, automorphic inductions and symmetric cube lifts (these are the automorphic reasons in this case) should be excluded to obtain large image. In fact, building in this result, the large image conjecture for genus two Siegel modular forms (i.e., which give rise to cohomological cuspidal automorphic representations of $\GSp_4(\A_\QQ)$) was proved by Kumar, Kumari and Weiss in \cite{KKW} \cite{Wei22}. Another case where the large image conjecture is a theorem, is when $\pi$ is a regular algebraic, polarizable, cuspidal automorphic representation of $\GL_3(\A_\QQ)$ \cite{DV04}.

When $F$ is a totally real field and $\pi$ is a regular algebraic, polarizable, cuspidal automorphic representation of $\GL_n(\A_F)$, it is well known that  the image of each residual Galois representation $\overline{\rho}_{\pi,\ell}$ associated to $\pi$ is contained in $\GO_n(\overline{\F}_\ell)$ or $\GSp_n(\overline{\F}_\ell)$ \cite{BLGGT14} \cite{CG13}. As with the irreducibility conjecture, even in this case, practically nothing is known about the large image conjecture when we do not impose local conditions on $\pi$ and $n$ is large ($n > 4$). In this paper, we prove some weak versions of the large image conjecture for residual Galois representations associated to regular algebraic, self-dual (then polarizable), cuspidal automorphic representations of $\GL_n(\A_F)$ when $n$ is a prime number $\leq 293$. As in \cite{Ze20}, the main limitation to extend our results to arbitrary $n$ is the lack of a sufficiently manageable classification of representations of finite simple groups of lie type of dimension $n$ (see Remark \ref{Rem293}).

More precisely, let $5 \leq n \leq 293$ be a prime different from 7, $F$ be a totally real field and $\overline{\rho}_{\pi,\ell}: G_F \rightarrow \GO_n(\overline{\F}_\ell)$ be a Galois representation associated to a regular algebraic, self-dual, cuspidal automorphic representation $\pi = \otimes_v ' \pi_v $ of $\GL_n(\A_F)$. In this case, we can deduce from the classification of the maximal subgroups of $\GO_n(\F_{\ell})$, that the only way for the image of $\overline{\rho}_{\pi,\ell}$ to be small (after excluding the cases when $\overline{\rho}_{\pi,\ell}$ is reducible or its image is bounded independent of $\ell$) is that it can be obtained as an $(n-1)$th-symmetric power (on the Galois side). In Theorem \ref{Prin}, we are able to exclude this possibility by imposing an appropriate condition on the weight of $\pi$ (then on the Hodge-Tate numbers) and we prove that the representations $\overline{\rho}_{\pi,\ell}$ have large image for all $\ell$ in a set of primes of positive Dirichlet density. Additionally, by imposing a local condition at some finite place $v$, we can prove a large image result for a set of primes of Dirichlet density one (Theorem \ref{vfue}). On the other hand, as $\GO_7$ contains the exceptional group $G_2$, the condition on the Hodge-Tate numbers mentioned above are not enough to prove a large image result when $n=7$. Instead, we prove large image results (Theorem \ref{Princ} and Theorem \ref{vfuer}) for certain automorphic $G_2$-valued Galois representations constructed by Chenevier in \cite{Ch19}. 

We remark that all our proofs are an exercise to explore what we can say about the large image conjecture by using only standard techniques and currently available tools: classification of the maximal subgroups of Lie type groups, Fontaine-Laffaille theory, Serre's modularity conjecture,  symmetric power functoriality for modular forms, known cases of Clozel's conjecture and irreducibility conjecture, and some currently available results about residual irreducibility of compatible systems of Galois representations.
In fact, when $F=\QQ$, thanks to the Serre's modularity conjecture and the symmetric power functoriality for modular forms, 
stronger results (on the automorphic side) can be obtained. For example, we verify that being a symmetric power lift is an automorphic reason for the image to be small (Theorem \ref{Prin1} and Theorem \ref{Prin2}) and we prove the large image conjecture for $n=5$ (Corollary \ref{lic}).


\section{Available tools in Galois representations}\label{sec:2}

In this section, we review some definitions and results about Galois representations associated to cuspidal automorphic representations of $\GL_n$ over totally real fields.  Our main references are \cite{BLGGT14}, \cite{Bar20}  and \cite{BG14}. 

Let $F$ be a totally real field and $\pi = \otimes_v ' \pi_v$ be a cuspidal automorphic representation of $\GL_n(\A_F)$. 
Let $v$ be an Archimedean place of $F$ and $\tau :F \hookrightarrow \CC$ be the embedding inducing $v$. 
The Langlands classification associates to $\pi_\tau = \pi_v$ a semi-simple representation $\phi_\tau: W_\RR \rightarrow \GL_n(\CC)$ of the Weil group $W_\RR$.
We will say that $\pi_\tau$ is \emph{algebraic},
 if the restriction of $\phi_\tau$ to the Weil group $W_\CC = \CC^\times$ is of the form 
\[
\phi_\tau \vert_{\CC^\times} = \chi_{\tau,1} \oplus \cdots \oplus \chi_{\tau,n}
\]
where  $\chi_{\tau,i}: \CC^\times \rightarrow \CC^\times$ are characters such that 
\[
\chi_{\tau, i}(z) = z^{a_{\tau,i}} \overline{z}^{b_{\tau,i}}
\] 
with $a_{\tau,i}, b_{\tau,i} \in \ZZ$. Note that our definition of algebraic representation is not Clozel's definition \cite{Clo90} and in fact corresponds to the notion of being $L$-algebraic in the sense of \cite{BG14}, which is more appropriate for the purposes of this paper. However, it is possible to recover Clozel's definition, which corresponds to $C$-algebraic representation in the sense of \cite{BG14}, by twisting the $\chi_{\tau, i}(z)$ by $\vert z \vert ^{\frac{n-1}{2}}$ (see \cite[Remark 1.8]{Shi20}). 

Let $\ZZ^n_+$ denote the set of $n$-tuples of integers $\{ \alpha_1, \ldots, \alpha_n \} \in \ZZ^n$ such that $\alpha_1 \leq \cdots \leq  \alpha_n$. After reordering indices, we will refer to the $n$-tuple 
$\{ a_{\tau,1}, \ldots a_{\tau,n} \} \in \ZZ^n_+$ 
as the weight of $\pi_\tau$. Moreover, we will say that $\pi_\tau$ is \emph{regular} if the $a_{\tau,i}$ are distinct. Finally, we will say that $\pi$ is \emph{regular algebraic} of weight 
$\{a_{\tau,1} , \ldots , a_{\tau,n} \}_\tau \in (\ZZ^n_+)^{\Hom_\QQ (F,\CC)}$,
if for each $\tau \in \Hom_\QQ (F, \CC)$, $\pi_\tau$ is algebraic and regular of weight $\{ a_{\tau,1}, \ldots , a_{\tau,n} \} \in \ZZ^n_+$.

Let $\pi = \otimes' _v \pi_v$ be a regular algebraic cuspidal automorphic representation of $\GL_n(\A_F)$ and $S_\pi$ be the finite set of finite places $v$ of $F$ at which $\pi_v$ is ramified. From now on, for each prime $\ell$, we fix an isomorphism $\iota_\ell: \overline{\QQ}_\ell \cong \CC$ and assume that $\pi$ is \emph{self-dual}, i.e., $\pi^\vee \simeq \pi$ (then it is polarizable). We remark that, when $n$ is odd, any regular algebraic, polarizable, cuspidal automorphic representation of $\GL_n(\A_F)$ is (up to a twist by an algebraic character) self-dual. Then, it can be proved that, for each prime $\ell$, there exists a continuous semi-simple representation
\begin{equation}\label{repre}
\rho_{\pi,\ell} : G_F \longrightarrow \GL_n(\overline{\QQ}_{\ell})
\end{equation}
such that if $v \notin S_\pi$ and $v \nmid \ell$ then $\rho_{\pi,\ell}$ is unramified at $v$ and the characteristic polynomial of a Frobenius element $\Frob_v$ satisfies
\[
\det(X - \rho _{\pi,\ell} (\Frob_v)) = \iota_\ell^{-1} \det (X-c(\pi_v)),
\]
where $c(\pi_v)$ is the Satake parameter of $\pi_v$ viewed as a semi-simple conjugacy class in $\GL_n(\CC)$; while if $v \vert \ell$ then $\rho_{\pi,\ell} \vert_{G_{F_v}}$ is de Rham and in fact crystalline when $v \notin S_\pi$ \cite[Theorem 3.1.2]{Ta16}. It is expected that the Galois representations $\rho_{\pi,\ell}$ are always irreducible \cite{Ram08}. However, the best-known result for general $n$ is the following  \cite[Theorem 1.7]{PT15}  \cite[Corollary B]{TY07}.

\begin{theorem}\label{irred}
Let $F$ be a totally real field and $\pi = \otimes' _v \pi_v$ be a regular algebraic, self-dual, cuspidal automorphic representation of $\GL_n(\A_F)$. Then, there exists a positive Dirichlet density set of primes $\mathcal{L}$ such that for all $\ell \in \mathcal{L}$ the Galois representation $\rho_{\pi,\ell}$ is irreducible. Moreover, if $\pi_v$ is square integrable for some finite place $v$, the Galois representations $\rho_{\pi,\ell}$ are irreducible for all primes $\ell$ such that $v \nmid \ell$.
\end{theorem}

Now, we will explain the relationship between the Hodge-Tate numbers of $\rho_{\pi,\ell} \vert_{G_{F_v}}$  and the inertial weights of its reduction modulo $\ell$. As $ \rho_{\pi,\ell} \vert_{G_{F_v}}$ is de Rham (in particular, it is Hodge-Tate) at $v \vert \ell$, for each embedding $\tau: F_v \hookrightarrow \overline{\QQ}_\ell$, we can attach to $ \rho_{\pi,\ell} \vert_{G_{F_v}}$ a multiset of integers 
$\HT_{\tau} ( \rho_{\pi,\ell} \vert_{G_{F_v}}) = \{\alpha_{\tau,1} , \ldots, \alpha_{\tau,n} \} \in \ZZ^n_+$
called the $\tau$-Hodge-Tate numbers of $ \rho_{\pi,\ell} \vert_{G_{F_v}}$. These numbers can be obtained from the weight of $\pi$ as follows. 

Let $\{a_{\tau,1} , \ldots , a_{\tau,n} \}_\tau \in (\ZZ^n_+)^{\Hom_\QQ (F,\CC)}$ be the weight of $\pi$. Identifying $\{ (v,\tau): v \vert \ell, \; \tau \in \Hom_{\QQ_\ell} (F_v, \overline{\QQ}_\ell) \}$ with $\Hom_{\QQ}(F, \CC)$ via the fixed isomorphism $\iota$, we have that
\[
\HT_{\tau} ( \rho_{\pi,\ell} \vert_{G_{F_v}}) = \{a_{\tau,1} , \ldots, a_{\tau,n} \} \in \ZZ^n_+,
\]
where $\tau$ in the right side is the embedding associate to the pair $(v,\tau)$. 

On the other hand, let $\overline{\rho}_{\pi,\ell} : G_{F} \rightarrow \GL_n(\overline{\F}_\ell)$ be the mod $\ell$ reduction of $\rho_{\pi,\ell}$ and $\F_v$ be the residue field of $F_v$.  If we assume that $F$ is unramified at $\ell$, we can attach to $\overline{\rho}_{\pi,\ell} \vert_{G_{F_v}}$ a subset 
\[
\Inert (\overline{\rho}_{\pi,\ell} \vert_{G_{F_v}}) \subset (\ZZ^n_+)^{\Hom_{\F_\ell}(\F_v,\overline{\F}_\ell)}
\]
called the \emph{set of inertial weights} of $\overline{\rho}_{\pi,\ell} \vert_{G_{F_v}}$. These weights, which are an analogue of Hodge-Tate numbers for mod $\ell$ representations, only depends on the restriction to the inertia subgroup $I_{F_v} \subset G_{F_v}$  of the semi-simplification of $\overline{\rho}_{\pi,\ell} \vert_{G_{F_v}}$.  Let
\[
\HT ( \rho_{\pi,\ell} \vert_{G_{F_v}}) = \{a_{\tau,1} , \ldots, a_{\tau,n} \}_\tau \in (\ZZ^n_+)^{\Hom_{\QQ_\ell}(F_v,\overline{\QQ}_\ell)}
\]
be the Hodge-Tate numbers of $\rho_{\pi,\ell} \vert_{G_{F_v}}$. Note that, as we are assuming that $F$ is unramified at $\ell$, we can index the Hodge-Tate numbers of $\rho_{\pi,\ell} \vert_{G_{F_v}}$ by embeddings $\F_v \hookrightarrow \overline{\F}_\ell$ rather than embeddings $F_v \hookrightarrow \overline{\QQ}_\ell$. Thus, we can see $\HT (\rho_{\pi,\ell} \vert_{G_{F_v}})$ as an element of $(\ZZ^n_+)^{\Hom_{\F_\ell}(\F_v,\overline{\F}_\ell)}$. By Fontaine-Lafaille theory \cite[Theorem 1.0.1]{Bar20} (see also \cite{FL82}) we have that, if $\rho_{\pi,\ell} \vert_{G_{F_v}}$ is crystalline and $a_{\tau,n} - a_{\tau,1} \leq \ell$ for all $\tau \in \Hom_{\QQ_\ell} (F_v, \overline{\QQ}_\ell)$, 
\begin{equation}\label{FLT}
\HT ( \rho_{\pi,\ell} \vert_{G_{F_v}}) \in \Inert (\overline{\rho}_{\pi,\ell} \vert_{G_{F_v}}).
\end{equation}

From now on, $\overline{\rho}_{\pi,\ell}$ will denote the semi-simplification of the mod $\ell$ reduction of $\rho_{\pi,\ell}$, which is usually called the \emph{residual representation} of $\rho_{\pi,\ell}$ for short.
The main goal of this paper is to study the image of $\overline{\rho}_{\pi,\ell}$, following the philosophy of \cite{Di02} \cite{DV04} \cite{DV08} \cite{DZ20} \cite{MS22} \cite{Rib75} \cite{Ri85} \cite{SwD72} \cite{Wei18}, i.e., by using a known case of the irreducibility conjecture (Theorem \ref{irred}), Fontaine-Lafaille theory (\ref{FLT}) and the classification of maximal subgroups of certain Lie type groups as will be described below.


\section{Study of the images I: $\GO_p$-valued representations}\label{sec:3}

As pointed out in the Introduction, self-duality implies that  the residual representations $\overline{\rho}_{\pi,\ell}:G_F \rightarrow \GL_n(\overline{\F}_\ell)$ factors through a map to $\GO_n(\overline{\F}_\ell)$ or to $\GSp_n(\overline{\F}_\ell)$ \cite[\S 2.2]{CG13}. In fact, as $G_F$ is compact $\overline{\rho}_{\pi,\ell}$ factors through $\GO_n(\F_q)$ or $\GSp_n(\F_q)$ for some finite field $\F_q$ of $q =\ell^s$ elements. Despite these facts, a complete understanding of the images of the residual representations $\overline{\rho}_{\pi,\ell}$ for general $n$ is still far from our reach. However, the problem is simplified substantially if we assume that $n$ is equal to an odd prime number $p$.  In that case, the representations $\overline{\rho}_{\pi,\ell}$ are $\GO_p$-valued and we are able to prove the following result.

\begin{theorem}\label{Prin}
Let $p\leq 293$ be an odd prime different from $3$ and $7$. Let $\pi$ be a regular algebraic, self-dual, cuspidal automorphic representation of $\GL_p(\A_F)$ and assume that the weight of $\pi$ is such that for some $\tau \in \Hom_{\QQ}(F, \CC)$
\begin{equation}\label{peso}
\{ a_{\tau, 1}, \ldots, a_{\tau, p} \} \neq \left\lbrace - \left( \tfrac{p-1}{2} \right) h, - \left( \tfrac{p-1}{2}-1 \right) h, \ldots , -h ,0, h , \ldots , \left( \tfrac{p-1}{2}-1 \right) h, \left( \tfrac{p-1}{2} \right) h \right\rbrace
\end{equation} 
for all $h \in \ZZ_{>0}$. Then, there exists a positive Dirichlet density set of primes $\mathcal{L}$ such that for all $\ell \in \mathcal{L}$ the image of $\overline{\rho}_{\pi,\ell}$ contains  the index two subgroup $\Omega_p(\F_\ell)$ of $\SO_p(\F_\ell)$ defined as the kernel of the spinor norm map. 
\end{theorem}

The proof of this theorem follows the structure of \cite{Di02}, \cite{DV04}, \cite{DV08} and \cite{DZ20}. Then, as in loc. cit., the proof of Theorem \ref{Prin} is done by considering the possible images of $\overline{\rho}_{\pi,\ell}$ given by the maximal subgroups of $\Omega_p(\F_\ell)$. 

One of the reasons why proving the large image conjecture looks very difficult for an arbitrary $n$, is because the classification of the maximal subgroups of $\Omega_n(\F_\ell)$ and $\Sp_n(\F_\ell)$ is very complicated due to the large number of Lie type groups that appear in it. However, when $n=p$, it turns out that the maximal subgroups of Lie type of $\Omega_p(\F_\ell)$ can be very cleanly classified into a few classes as we will explain now.
 
Let $p$ and $\ell$ be two odd primes. In \cite{As84} (see also \cite{BHRD13} and \cite{KL90}), Aschbacher proved that if $G$ is a maximal subgroup of $\Omega_p(\F_\ell)$ then one of the following holds: 
\begin{enumerate}
\item $G$ is a reducible group;
\item $G \cong 2^{p-1}. A_p$, if $\ell \equiv \pm 3 \mod 8$;
\item $G \cong 2^{p-1}. S_p$, if $\ell \equiv \pm 1 \mod 8$; or
\item $G$ is of class $\mathcal{S}$.
\end{enumerate}
The subgroups in cases i, ii and iii of Aschbacher's classification above (called of \emph{geometric type}) are well described across all dimensions. In fact, a much more general result, giving a description of the maximal subgroups of geometric type of all the finite classical groups, was proven by Aschbacher in \cite{As84}. 
In contrast, the subgroups of class $\mathcal{S}$ are not susceptible to a uniform description across all dimensions. However, we can obtain a partial description of these subgroups when the dimension is a prime number. 

Recall that a maximal subgroup $G$ of $\Omega_p(\F_\ell)$ is \emph{of class} $\mathcal{S}$ if: 
\begin{enumerate}
\item $G^\infty$ acts absolutely irreducibly on $\F_\ell^p$; and
\item P$G$ is \emph{almost simple}, i.e., $S \leq \mbox{P}G \leq \Aut (S)$ for some non-abelian simple group $S$.
\end{enumerate}
Here, P$G$ denotes the projectivization of $G$, which is defined as the image of $G$ in $\PGL_p(\F_\ell)$, and $\displaystyle G^\infty := \bigcap _{i \geq 0} G^{(i)}$, where $G^{(i)}$ denotes the $i$-th derived subgroup of $G$.  

In this paper, by a \emph{finite simple group of Lie type}, we mean a finite twisted or non-twisted simple adjoint Chevalley group in characteristic $\ell \neq 2,3$.  This kind of groups includes $G_2(\F_\ell)$ (the automorphism group of the octonion algebra over $\F_\ell$), which appears in the following result.

\begin{proposition}\label{tipS}
Let $\ell \neq 2,3$, $p \geq 5$ and $G$ be a maximal subgroup of $\Omega_p(\F_\ell)$ of class $\mathcal{S}$. Then, if $\vert \mbox{P}G \vert > p^4(p+2)$, the non-abelian simple group $S$ is a finite simple group of Lie type in characteristic $\ell$. In fact, if $p=7$, $S \cong G_2(\F_\ell)$ and, if $7 \neq p \leq 293$, $S \cong \PSL_2(\F_\ell)$.
\end{proposition} 
\begin{proof}
The proof of the first part is very similar to \cite[Proposition 3.3]{Ze20}. The second part follows from \cite[Table 8.40]{BHRD13} and the tables in \cite{Lu01}  by using the same technique as in \cite[\S 5]{Ze20}.

\end{proof}

\begin{remark}\label{Rem293}
We remark that $\Omega_3(\F_\ell)$ does not have maximal subgroups of class $\mathcal{S}$ with $S$ a finite group of Lie type  when $\ell \geq 5$ \cite[Table 8.7]{BHRD13}. Then, the classification of maximal subgroups of $\Omega_3(\F_\ell)$  is slightly different  from the classification for $\Omega_p(\F_\ell)$ with $p>3$.
Fortunately, when $p=3$, the large image conjecture is essentially a theorem that follows from \cite{CG13}\cite{Hui23} and \cite{DV04}.
We also remark that the above proposition is expected to be true and uniform for all primes $p>7$. The only limitation (as in \cite{Ze20}) to prove such result for $p>293$, is the lack of a complete classification of representations of finite simple groups of Lie type as in \cite{Lu01}. Accordingly, any improvement in the bounds in such classification should imply an automatic improvement in several of our results. 
\end{remark}

Now, we are ready to prove our theorem. As we mentioned above, the proof follows the structure of \cite{Di02} \cite{DV04} \cite{DV08} and \cite{DZ20}. Then, as in loc. cit., it will be given by showing that the image of $\overline{\rho}_{\pi,\ell}$ is not contained in any subgroup lying in cases i, ii, iii and iv of Aschbacher's classification.

\paragraph{\emph{Proof of Theorem} \ref{Prin}}
First, by Theorem \ref{irred}, we have that there exists a positive  Dirichlet density set of primes $\mathcal{L}''$ such that for all $\ell \in \mathcal{L}''$ the representation $\rho_{\pi,\ell}$ is irreducible. 
Then, by \cite[Corollary 1.3]{PSW18}, we have that there is a positive Dirichlet density set of primes $\mathcal{L}' \subset \mathcal{L}''$ (obtained by removing a finite number of primes from $\mathcal{L}''$) such that $\overline{\rho}_{\pi,\ell}$ is irreducible for all $\ell \in \mathcal{L}'$. Consequently, if $\ell \in \mathcal{L}'$, the image of $\overline{\rho}_{\pi,\ell}$ cannot be contained in a maximal subgroup in case i.

Now, we will deal with case iv, when  $\vert PG_\ell \vert > p^4(p+2)$. Here, $G_\ell$ denotes the image of $\im (\overline{\rho}_{\pi,\ell})$. In this case, $S=\PSL_2(\F_\ell)$ by Proposition \ref{tipS}, and $\PSL_2(\F_\ell)$ fits into $\POM_p(\F_\ell)$ via $\Sym^{p-1}: \PSL_2 \rightarrow \POM_p$. 
Then, if $G_\ell$ is contained in $\Sym^{p-1} (\PSL_2(\F_\ell))$, the elements of $G_\ell$ are of the form
\[
\Sym^{p-1}
\begin{pmatrix}
x & \bullet \\
\bullet & y 
\end{pmatrix} = \begin{pmatrix}
x^{p-1} & * & * & * & * & * & * \\
* & x^{p-2} y & * & * & * & * & * \\
. & . & . & . & . & . & . \\
. & . & . & . & . & . & . \\
. & . & . & . & . & . & . \\
* & * & * & * & * & x y^{p-2} & * \\
* & * & * & * & * & * & y^{p-1} 
\end{pmatrix}
\]
where $x,y \in \overline{\F}_\ell$. Then, we can deduce that
\[
(x^{((p-1)-m)} y^m)(x^{((p-1)-m)-2} y^{m+2}) = (x^{((p-1)-m)-1} y^{m+1})^2
\]
for $0 \leq m \leq p-3$. From these equalities, we have that for all $\ell$ sufficiently large and any $v \vert \ell$, the inertial weights $\{ \alpha_{\tau,1}, \ldots, \alpha_{\tau,p} \}_\tau \in \Inert(\overline{\rho}_{\pi, \ell} \vert _{G_{F_v}})$ should satisfy the following relation
\begin{equation}\label{pesos}
\alpha_{\tau,i}+\alpha_{\tau, i+2} = 2 \alpha_{\tau, i+1}
\end{equation}
for $1 \leq i \leq p-2$. 

On the other hand, by the self-duality hypothesis in $\pi$, we have that for each finite place $v \vert \ell$ and each $\tau \in \Hom_{\QQ_\ell} (F_v, \overline{\QQ}_\ell)$, the $\tau$-Hodge-Tate numbers of $\rho_{\pi, \ell} \vert _{G_{F_v}}$ are of the form:
\[
\HT_{\tau} (\rho_{\pi, \ell} \vert _{G_{F_v}}) = 
\left\lbrace - h_{\tau, \frac{p-1}{2} }, -h_{\tau, \frac{p-1}{2}-1 }, \ldots, -h_{\tau,1} ,0, h_{\tau,1} , \ldots , h_{\tau, \frac{p-1}{2}-1 }, h_{\tau, \frac{p-1}{2}} \right\rbrace \in \ZZ^p_+
\]
Moreover, if $\ell$ is such that $v \notin S_\pi$ for $v\vert \ell$ and $2 h_{\tau, \frac{p-1}{2}} \leq \ell$, then $\rho_{\pi, \ell} \vert _{G_{F_v}}$ is crystalline and by (\ref{FLT})
\[
\left\lbrace - h_{\tau, \frac{p-1}{2} }, -h_{\tau, \frac{p-1}{2}-1 }, \ldots, -h_{\tau,1} ,0, h_{\tau,1}, \ldots , h_{\tau, \frac{p-1}{2}-1 }, h_{\tau, \frac{p-1}{2}} \right\rbrace \in \Inert (\overline{\rho}_{\pi,\ell} \vert_{G_{F_v}})
\]
Therefore, in this case, the $\tau$-Hodge-Tate numbers $\HT_{\tau} (\rho_{\pi, \ell} \vert _{G_{F_v}})$ should satisfy (\ref{pesos}). However, that only happens if 
$h_{\tau,j} = j h$ for all $j \in \{0,1, \ldots, \frac{p-1}{2}\}$
and any $\tau \in \Hom_{\QQ_\ell} (F_v, \overline{\QQ}_\ell)$. 
Then, by our assumption on the weight of $\pi$ (\ref{peso}), we have that the image of $\overline{\rho}_{\pi, \ell}$ cannot be contained in $\Sym^{p-1} (\PSL_2(\F_q))$ when $\ell$ is sufficiently large.

Finally, if the image of $\overline{\rho}_{\pi,\ell}$ is contained in one of the maximal subgroups in cases ii, iii and iv with  $\vert PG_\ell \vert \leq p^4(p+2)$, the order of $G_\ell$ is bounded independently of $\ell$. Then, by \cite[Lemma 5.3]{CG13}, if $\ell$ is large enough, the image of $\overline{\rho}_{\pi,\ell}$ cannot be contained in one of these maximal subgroups.

Therefore, there is a positive Dirichlet density set of primes $\mathcal{L}$ (obtained after removing possibly a finite number of small primes from $\mathcal{L'}$) such that for all $\ell \in \mathcal{L}$ the image of $\overline{\rho}_{\pi,\ell}$ contains $\Omega_p(\F_\ell)$.

\qed

\bigskip

We remark that, if we allow certain local ramification behavior in our automorphic representations, we can obtain the following strong version of Theorem \ref{Prin}.

\begin{theorem}\label{vfue}
Let $p\leq 293$ be an odd prime different from $3$ and $7$. Let $\pi = \otimes' _v \pi_v$ be a regular algebraic, self-dual, cuspidal automorphic representation of $\GL_p(\A_F)$ satisfying condition (\ref{peso}) of Theorem \ref{Prin} and assume that for some finite place $v$, $\pi_v$ is square integrable. Then there exists a set of primes $\mathcal{L}$ of Dirichlet density one such that for all  $\ell \in \mathcal{L}$ the image of $\overline{\rho}_{\pi,\ell}$ contains $\Omega_p(\F_\ell)$.
\end{theorem}

\begin{proof}
Let $\ell$ be a rational prime such that $v \nmid \ell$.
As we are assuming that $\pi_v$ is square integrable, from Theorem \ref{irred}, we have that $\rho_{\pi,\ell}$ is irreducible. 
By Corollary 1.3 of \cite{PSW18}, there exists a set of primes $\mathcal{L}'$ of Dirichlet density one, such that for all $\ell \in \mathcal{L}'$, $\overline{\rho}_{\pi,\ell}$ is irreducible. The rest of the proof is exactly the same as the proof of Theorem \ref{Prin}.
In particular, the set of primes $\mathcal{L}$ of Dirichlet density one is obtained by removing a finite number of primes from $\mathcal{L'}$ as in the proof of Theorem \ref{Prin}.

\end{proof}

When $F=\mathbb{Q}$, examples of cuspidal automorphic representations satisfying the assumptions of Theorem \ref{Prin} can be obtained from the computations of Ta\"ibi \cite{Ta17}. 

More precisely, let $h_1, \ldots , h_{\frac{p-1}{2}} \in \mathbb{Z}$, with $h_1> \cdots > h_{\frac{p-1}{2}}>0$, and $O_{o}(h_1, \ldots , h_{\frac{p-1}{2}})$ be the set of regular algebraic, self-dual, cuspidal automorphic representations of $\text{GL}_p(\mathbb{A}_\mathbb{Q})$ of level one (i.e., everywhere unramified) and weight $\{-h_1, \ldots , -h_{\frac{p-1}{2}},0, h_{\frac{p-1}{2}}, \ldots ,h_1  \} \in \mathbb{Z}^p_+$. It follows from \cite[Theorem 1]{HC68} that the cardinality of $O_{o}(h_1, \ldots , h_{\frac{p-1}{2}})$ is finite. 

In the extended tables of \cite{Ta17}, Ta\"ibi computes explicitly the cardinality of these sets for $p=5,7,11,13$ and $h_1 \leq 49, 23, 16, 15$, respectively.
Then, thanks to these computations, we know that there exist cuspidal automorphic representations of $\text{GL}_p(\mathbb{A}_\mathbb{Q})$ of level one satisfying Theorem \ref{Prin} for $p=5,11,13$. For example, it follows from the tables that there are 37 regular algebraic, self-dual, cuspidal automorphic representations $\pi$ of $\text{GL}_{11}(\mathbb{A}_\mathbb{Q})$ of weight $\{ -16, -13, -11, -8, -3, 0,  3, 8, 11, 13, 16\}$ satisfying Theorem \ref{Prin}. In fact, all cuspidal automorphic representations found by Ta\"ibi for $\text{GL}_{13}(\mathbb{A}_\mathbb{Q})$ satisfy Theorem \ref{Prin}. We remark that, in Ta\"ibi's tables, there are also automorphic representations which do not satisfy Theorem \ref{Prin}. For example, there are 8 of weight $\{-15, -12, -9, -6, -3,0, 3, 6,9,12,15 \}$ for $\text{GL}_{11}(\mathbb{A}_\mathbb{Q})$ and 84 of weight $\{-46, -23, 0, 23, 46 \}$ for $\text{GL}_5(\mathbb{A}_\mathbb{Q})$.

Finally, we remark that, when $p=7$, we cannot ensure that the image of $\overline{\rho}_{\pi,\ell}$ contains $\Omega_7(\mathbb{F}_\ell)$ assuming only the hypotheses (\ref{peso}) of Theorem \ref{Prin}. However, as we will see in the next section, we are able to guarantee that the images of certain $G_2$-valued Galois representations contain $G_2(\mathbb{F}_\ell)$.


\section{Study of the images II: $G_2$-valued representations}\label{sec:4}

In this section, we deal with the 7-dimensional case omitted in the previous section. As can be seen in Proposition \ref{tipS}, the main difference with respect to the other dimensions is the occurrence of a new maximal subgroup of class $\mathcal{S}$ in $\Omega_7(\F_\ell)$, namely $G_2(\F_\ell)$, that requires special treatment.
We will start by showing that, indeed, there exists regular algebraic, self-dual, cuspidal automorphic representations of $\GL_7(\A_F)$ such that their $\ell$-adic and residual Galois representations associated to them, are $G_2$-valued.

Let $G_2$ be the automorphism group scheme of the standard split octonion algebra over $\ZZ$.  It is well known that, for any algebraically closed field $k$ of characteristic 0, there is a unique (up to isomorphism) irreducible $k$-linear algebraic representation $\sigma : G_2(k) \rightarrow \GL_7(k)$. Using (\ref{repre}), Chenevier \cite[Corollary 6.5, Corollary 6.10]{Ch19} proved the following result:

\begin{theorem}\label{Chen}
Let $\pi = \otimes'_v \pi_v$ be a regular algebraic, self-dual, cuspidal automorphic representation of $\GL_7(\A_F)$ and
assume that, for almost all finite places $v \notin S_\pi$, the Satake parameter $c(\pi_v)$ of $\pi_v$ is the conjugacy class of an element in $\sigma(G_2(\CC))$.
Then, for each prime $\ell$, there exists a continuous semi-simple representation
\[
\rho_{\pi,\ell} : G_{F} \longrightarrow G_2(\overline{\QQ}_{\ell})
\]
such that, if $v \notin S_\pi$ and $v \nmid \ell$ then $\rho_{\pi,\ell}$ is unramified and
\[
\det(X-\sigma(\rho_{\pi,\ell} (\Frob_v))) = \iota \det (X-c(\pi_v)),
\]
while if $v \vert \ell$ then $\rho_{\pi,\ell} \vert_{G_{F_v}}$ is de Rham and in fact crystalline when $v \notin S_\pi$. 
Moreover, the weight of $\pi$ is of the form:
\begin{equation}\label{pesg}
\{ -(h_\tau + k_\tau), -k_\tau, -h_\tau, 0, h_\tau,  k_\tau, h_\tau + k_\tau \}_\tau \in (\ZZ^7_+)^{ \Hom_{\QQ}(F, \CC)}.
\end{equation}
\end{theorem}

Following the line of reasoning of the previous section, we prove the following result for $G_2$-valued residual representations.

\begin{theorem}\label{Princ}
Let $\pi$ be a regular algebraic, self-dual, cuspidal automorphic representation of $\GL_7(\A_F)$ as in Theorem \ref{Chen} and assume that the weight of $\pi$ is such that $k_\tau \neq 2h_\tau$ for some $\tau \in \Hom_{\QQ}(F, \CC)$. Then, there exists a positive Dirichlet density set of primes $\mathcal{L}$ such that for all $\ell \in \mathcal{L}$ the image of $\overline{\rho}_{\pi,\ell}$ contains $G_2(\F_\ell)$.
\end{theorem}

\begin{proof}
Similarly to the proof of Theorem \ref{Prin}, to prove this theorem, we need to know the maximal subgroups of $G_2(\F_\ell)$. In \cite{Kle88} (see also Table 8.41 in \cite{BHRD13}), Kleidman proved that, if $\ell >3$, the maximal subgroups of $G_2(\F_\ell)$ are as follows:
\begin{enumerate}
\item reducible groups;
\item $\PGL_2(\F_\ell)$, $\ell \geq 11$;
\item $2^{3.}\PSL_3(\F_2)$, $\PSL_2(\F_{13})$, $\PSL_2(\F_8)$, $G_2(\F_2)$ and the Janko group $J_1$.
\end{enumerate}

Since $\rho_{\pi,\ell}$ is a Galois representation associated to a regular algebraic, self-dual, cuspidal automorphic representation $\pi$ of $\GL_7(\A_F)$, we can proceed as in Theorem \ref{Prin} to prove that there exists a positive Dirichlet density set of primes $\mathcal{L}'$ such that $\overline{\rho}_{\pi,\ell}$ is irreducible for all $\ell \in \mathcal{L}'$.
Then, if $\ell \in \mathcal{L}'$, the image of $\overline{\rho}_{\pi,\ell}$ cannot be contained in a reducible group (case i of Kleidman's classification above).

Now, to deal with case ii, note that $\PGL_2(\F_\ell)$ fits into $G_2(\F_\ell)$ via $\Sym^6: \PGL_2 \rightarrow G_2$ and that the weight of $\pi$ is of the form $\{ -(h_\tau + k_\tau), -k_\tau, -h_\tau, 0, h_\tau,  k_\tau, h_\tau + k_\tau \}\in \ZZ^7_+$ (\ref{pesg}) for each $\tau \in \Hom_\QQ (F,\CC)$. Then, as we are assuming that $k_\tau \neq 2h_\tau$ for some $\tau \in \Hom_\QQ (F,\CC)$, the weight of  $\pi$ in that $\tau$, is different from $\{ -3h_\tau, -2h_\tau, -h_\tau, 0, h_\tau,  2h_\tau, 3h_\tau \}$. Thus, we can also proceed as in Theorem \ref{Prin}, to prove that the image of $\overline{\rho}_{\pi, \ell}$ cannot be contained in $\Sym^6 (\PGL_2(\F_\ell))$ when $\ell$ is sufficiently large.

Finally, if the image of $\overline{\rho}_{\pi,\ell}$ is contained in one of the maximal subgroups in case iii, its order is bounded independently of $\ell$. Then, by \cite[Lemma 5.3]{CG13}, the image of $\overline{\rho}_{\pi,\ell}$ cannot be contained in one of these maximal subgroups if $\ell$ is sufficiently large. 
Hence, as in Theorem \ref{Prin}, we can obtain a positive Dirichlet density set of primes $\mathcal{L}$ (maybe after removing a finite number of small primes from $\mathcal{L'}$) such that for all $\ell \in \mathcal{L}$ the image of $\overline{\rho}_{\pi,\ell}$ contains $G_2(\F_\ell)$.

\end{proof}

As in the previous section, if $F=\QQ$, examples of cuspidal automorphic representations satisfying the assumptions of the theorem above can be obtained from the computations of Ta\"ibi \cite{Ta17} (see also \cite[Table 11]{CR15}).
Let $h,k \in \ZZ$, with $k>h>0$, and $\mathcal{G}_{2}(k,h)$ be the subset of cuspidal automorphic representations in $O_{o}(h+k,k,h)$ satisfying the assumptions of Theorem \ref{Princ}. Recall that in the extended tables of \cite{Ta17}, Ta\"ibi compute explicitly the cardinality of $O_{o}(h+k,k,h)$ for $h+k \leq 23$. 
Then, from Ta\"ibi's computations and Theorem 6.12 of \cite{Ch19}, we have that if 
\[
(k,h) \in \{ (8,5), (10,3), (9,5),(10,4),(12,2),(8,7),(11,4),(14,1), (16,1), (17,1) \},
\]
$\vert \mathcal{G}_{2}(k,h) \vert = 1$. In particular, this provides us with 10 examples of regular algebraic, self-dual, cuspidal automorphic representations of $\GL_7(\A_\QQ)$ satisfying Theorem \ref{Princ}.

On the other hand, if we allow certain local ramification behavior in $\pi$, we can obtain a strong version of Theorem \ref{Princ} in the same spirit as Theorem \ref{vfue}.

\begin{theorem}\label{vfuer}
Let $\pi = \otimes' _v \pi_v$ be a regular algebraic, self-dual, cuspidal automorphic representation of $\GL_7(\A_F)$ as in Theorem \ref{Princ} and assume that for some finite place $v$, $\pi_v$ is square integrable. Then there exists a set of primes $\mathcal{L}$ of Dirichlet density one such that for all  $\ell \in \mathcal{L}$ the image of $\overline{\rho}_{\pi,\ell}$ contains  $G_2(\F_\ell)$.
\end{theorem}

\begin{proof}
As $\rho_{\pi,\ell}$ is a Galois representations associated to a regular algebraic, self-dual, cuspidal automorphic representation $\pi$ of $\GL_7(\A_F)$ and $\pi_v$ is square integrable for some $v \nmid \ell$, we have that $\rho_{\pi,\ell}$ is irreducible  by Theorem \ref{irred}. As in Theorem \ref{vfue}, Corollary 1.3 of \cite{PSW18} implies that there exists a set of primes $\mathcal{L}'$ of Dirichlet density one, such that for all $\ell \in \mathcal{L}'$, $\overline{\rho}_{\pi,\ell}$ is irreducible. Then, the rest of the proof is exactly the same as the proof of Theorem \ref{Princ}. In particular, the set of primes $\mathcal{L}$ of Dirichlet density one is obtained by removing a finite number of primes from $\mathcal{L'}$ as in the proof of Theorem \ref{Princ}.

\end{proof}

Finally, we remark that Magaard and Savin \cite{MS22} have used this kind of local behavior to construct a self-dual cuspidal automorphic representation $\pi$ of $\GL_7(\A_\QQ)$ (unramified outside $5$ and such that $\pi_5$ is Steinberg), such that the images of the residual representations $\overline{\rho}_{\pi,\ell}:G_\QQ \rightarrow \GL_7(\overline{\F}_\ell)$ associated to $\pi$ are equal to $G_2(\F_\ell)$ for an explicit set of primes of Dirichlet density one.


\section{Improvements when $F=\QQ$}\label{sec:5}

In this section, by using  Serre's modularity conjecture, symmetric power functionality and a recent result on residual irreducibility of Galois representations due to Hui \cite{Hui22}, we perform some improvements to the results given in the previous sections when $F = \QQ$.

Let $\varpi$ be a regular algebraic cuspidal automorphic representation of $\GL_2(\A_\QQ)$ (which in fact corresponds to a twist of a cuspidal Hecke eigenform of weight at least 2 by regularity).
Let $p$ be an odd prime. By symmetric power functoriality for modular forms \cite{NT21}, the $(p-1)$th-symmetric power lifting $\Sym^{p-1}(\varpi)$  is a regular algebraic, self-dual, cuspidal automorphic representation of $\GL_p(\A_\QQ)$. 
Then, we will say that a regular algebraic, self-dual, cuspidal automorphic representation $\pi$ of $\GL_p(\A_\QQ)$ is a \emph{$(p-1)$th-symmetric power lift} if there is a regular algebraic cuspidal automorphic representation $\varpi$ of $\GL_2(\A_\QQ)$  such that, for any prime $\ell$,
\[
\rho_{\pi,\ell} \cong \Sym^{p-1} (\sigma_{\varpi,\ell}),
\]
where $\sigma_{\varpi,\ell}: G_\QQ \rightarrow \GL_2(\overline{\QQ}_\ell)$ is the $\ell$-adic Galois representation associated to $\varpi$ by Deligne \cite{Del71}. We remark that if $\pi$ is a $(p-1)$th-symmetric power lift, the weight of $\pi$ must be of the form 
\[
\left\lbrace - \left( \tfrac{p-1}{2} \right) h, - \left( \tfrac{p-1}{2}-1 \right) h, \ldots , -h ,0, h , \ldots , \left( \tfrac{p-1}{2}-1 \right) h, \left( \tfrac{p-1}{2} \right) h \right\rbrace .
\]
Thus, the automorphic representations considered in Theorem \ref{Prin} and Theorem \ref{Princ} cannot be $(p-1)$th-symmetric power lifts. 

Thanks to Serre's modularity conjecture, which is a theorem when $F = \QQ$ (see \cite{KW09a}, \cite{KW09b} and \cite{Di12}), we can prove the following result.

\begin{theorem}\label{Prin1}
Let $p\leq 293$ be an odd prime different from $3$ and $7$. Let $\pi = \otimes' _v \pi_v$ be a regular algebraic, self-dual, cuspidal automorphic representation of $\GL_p(\A_\QQ)$. If $\pi$ is not a $(p-1)$th-symmetric power lift, then there exists a positive Dirichlet density set of primes $\mathcal{L}$ such that for all $\ell \in \mathcal{L}$ the image of $\overline{\rho}_{\pi,\ell}$ contains $\Omega_p(\F_\ell)$.
Moreover, if at some finite place $v$, $\pi_v$ is square integrable, $\mathcal{L}$ has Dirichlet density one.
\end{theorem}

\begin{proof}
As in Theorem \ref{Prin}, the proof is given by showing that the image of $\overline{\rho}_{\pi,\ell}$ cannot be contained in any subgroup lying in cases i $-$ iv of Aschbacher's classification.

From now on, we will assume that the weight of $\pi$ is of the form
$\left\lbrace - \left( \tfrac{p-1}{2} \right) h, \ldots , -h ,0, h , \ldots ,  \left( \tfrac{p-1}{2} \right) h \right\rbrace$. 
The other weights were dealt in Theorem \ref{Prin}.
Note that the case i of Aschbacher's classification can be dealt in exactly the same way as in the proof of Theorem \ref{Prin}. Then, there is a positive Dirichlet density set of primes $\mathcal{L'}$ such that, for all $\ell \in \mathcal{L}'$, $\overline{\rho}_{\pi,\ell}$ is irreducible. 

Now, let $\ell \in \mathcal{L}'$ and assume that the image of $\overline{\rho}_{\pi,\ell}$ is contained in a maximal subgroup lying in case iv of Aschbacher's classification. Then
\[
\overline{\rho}_{\pi, \ell} \simeq \Sym^{p-1}(\overline{\sigma}_\ell),
\]
where $\overline{\sigma}_\ell: G_\QQ \rightarrow \GL_2(\overline{\F}_\ell)$ is a two-dimensional irreducible Galois representation. From \cite[Proposition 1]{Tay12}(see also \cite{Ta16} and \cite{CH16}), we have that, if  $c \in G_\QQ$ is a complex conjugation then $ \Tr (\rho_{\pi, \ell}(c)) = \pm 1$. Thus, by the structure of $\Sym^{p-1}$, we have that $\overline{\sigma}_\ell$ is odd. Hence, from the explicit description of $\Sym^{p-1}$, the form of the Hodge-Tate numbers of $\rho_{\pi,\ell} \vert _{G_{\QQ_\ell}}$ and Serre's modularity conjecture; we have that, if  $\ell \notin S_\pi$ and $(p-1)h \leq \ell$, there exists a cuspidal Hecke eigenform $f$ of weight $h+1\geq 2$ and level bounded independently of $\ell$, such that

\begin{equation}\label{simi}
\rho_{\pi, \ell} \equiv \Sym^{p-1}(\sigma_{\varpi_{f},\ell}) \mod \ell,
\end{equation}
where $\varpi_f$ is the regular algebraic cuspidal automorphic representation of $\GL_2(\A_\QQ)$ corresponding to $f$. We remark that the set of such cuspidal Hecke eigenform (with a fixed weight and bounded level) is finite.
Then, if the congruence (\ref{simi}) is satisfied for infinitely many primes $\ell$, by Dirichlet principle, we have that there is a fixed cuspidal Hecke eigenform $f$ such that
\[
\rho_{\pi, \ell} \equiv \Sym^{p-1}(\sigma_{\varpi_f,\ell}) \mod \ell,
\]
for infinitely many primes $\ell$. Therefore, by Chebotarev's density theorem, it follows that 
\[
\rho_{\pi, \ell} \simeq \Sym^{p-1}(\sigma_{\varpi_f,\ell}),
\]
for all primes $\ell$. Thus, $\pi$ is a $(p-1)$th-symmetric power lift, contradicting our assumption on $\pi$.

Finally, the rest of the cases of Aschbacher's classification can also be dealt as in the proof of Theorem \ref{Prin} and the set of primes $\mathcal{L} \subset \mathcal{L}'$ of positive Dirichlet density can be obtained by removing at most a finite number of small primes from $\mathcal{L'}$. Moreover, if $\pi_p$ is square integrable for some finite place $v$, we can proceed exactly as in Theorem \ref{vfue}.

\end{proof}

Similarly, we can prove the following result concerning to 7-dimensional $G_2$-valued Galois representations such as those studied in Section \ref{sec:4}.

\begin{theorem}\label{Prin2}
Let $\pi = \otimes' _v \pi_v$ be a regular algebraic, self-dual, cuspidal automorphic representation of $\GL_7(\A_\QQ)$ as in Theorem \ref{Chen}. If $\pi$ is not a sixth symmetric power lift then there exists a positive Dirichlet density set of primes $\mathcal{L}$ such that for all $\ell \in \mathcal{L}$ the image of $\overline{\rho}_{\pi,\lambda}$ contains $G_2(\F_{\ell})$.
Moreover, if at some finite place $v$, $\pi_v$ is square integrable, then $\mathcal{L}$ has Dirichlet density one.
\end{theorem}

\begin{remark}
From the computations of Ta\"ibi \cite{Ta17}, we can obtain an example of cuspidal automorphic representation satisfying Theorem \ref{Prin2} but not Theorem \ref{Princ}. More precisely, from Theorem 6.12 of \cite{Ch19} we have that there exists a regular algebraic, self-dual, cuspidal automorphic representation $\pi$ of $\GL_7(\A_\QQ)$ satisfying Theorem \ref{Chen}, but contained in $\mathcal{G}_2(8,4)$. Then, it does not satisfy Theorem \ref{Princ}. 
However, this cuspidal automorphic representation satisfies Theorem \ref{Prin2} because it is not a sixth symmetric power lift. Indeed, if it were a sixth symmetric power lift, by the discussion in the proof of Theorem \ref{Prin1}, it should come from a cuspidal Hecke eigenform $f$ of weight $5$ and level 1, which does not exist. 
\end{remark}

As we mentioned above, large image conjecture is a theorem for $n=3$ due to Dieulefait and Vila \cite{DV04}. We end this paper by proving the analogous result for $n=5$, which is a simple consequence of  our previous results and the recent work of Hui \cite{Hui22}.

\begin{corollary}\label{lic}
Let $\pi$ be a regular algebraic, self-dual, cuspidal automorphic representation of $\GL_5(\A_\QQ)$, which is not a fourth symmetric power lift. Then, the image of $\overline{\rho}_{\pi,\ell}$ contains $\Omega_5(\F_\ell)$ for almost all $\ell$.
 \end{corollary}

\begin{proof}
In this case, residual irreducibility of $\overline{\rho}_{\pi,\ell}$ is known for almost all $\ell$ \cite[Theorem 1.4]{Hui22}. Then, the image of $\overline{\rho}_{\pi,\ell}$ cannot be contained in a maximal subgroup in case i of Aschbacher's classification for almost all $\ell$. The rest of maximal subgroups of $\Omega_5(\F_\ell)$ can be dealt exactly the same as in the proof of Theorem \ref{Prin}.

\end{proof}

\subsection*{Acknowledgments}

I would like to thank the organizers of the BIRS-CMO workshop ``Langlands Program: Number Theory and Representation Theory", held in Casa Matemática Oaxaca (CMO), for inviting me to share part of my research and their encouragement to write this paper. Some results in this paper were proved during a FONDECYT postdoctoral stay at the Instituto de Matemáticas PUCV and a CONACYT postdoctoral stay at the Unidad Cuernavaca del Instituto de Matemáticas UNAM. 
In particular, some results of my old unpublished paper arXiv:2008.00556 are included in this work. I sincerely thank these institutions for their support and the optimal working conditions. I also thank Ariel Weiss for very useful comments and suggestions on a previous version of this paper. Finally, I want to give special thanks to the referee, whose suggestions and comments have greatly improved the presentation of this paper.

The author was partially supported by the ANID Proyecto FONDECYT Postdoctorado No. 3190474, Estancias Posdoctorales por México 2021 - Modalidad Académica CONACYT and the CONAHCYT grant CBF2023-2024-224.

\noindent 
\textsc{Centro de Investigación en Matemáticas, A.C. \\ Jalisco s/n. Col. Valenciana \\
 36023 Guanajuato, Mexico.}\\

\newpage

\noindent \textit{Current Address:}

\smallskip

\noindent 
\textsc{Universidad Autónoma Metropolitana - Unidad Iztapalapa \\
Departamento de Matemáticas, Edificio Anexo al T\\
Av San Rafael Atlixco No.186 Col.Vicentina \\
C.P. 09340, Iztapalapa, Mexico City, Mexico.}\\
E-mail: \texttt{adrian.zenteno@xanum.uam.mx}

\end{document}